\documentclass[]{amsart}
\usepackage{amsmath,amsfonts,amssymb}
\usepackage{enumerate}
\usepackage{xspace}
\usepackage{xcolor}
\usepackage{graphicx}
\usepackage{bbm}
\usepackage{subcaption}
\usepackage{hyperref}
\usepackage[hyperref,backref=true,backend=biber,style=numeric-comp,citestyle=numeric,isbn=false,url=false,eprint=false,giveninits=true]{biblatex}

\graphicspath{{./Figs/}}
\newcommand\numberthis{\addtocounter{equation}{1}\tag{\theequation}}

\def \Rd{\mathbb{R}^d}
\def \R {\mathbb{R}}

\def \xc {X_i}

\def \Projs {\Pi_i ^s }

\def \xlbc{{Y}^L_i}

\def \flq{{F}^{L}_i}
\def \ProjL {\Pi_i ^{L_i}}

\def \Projs {\Pi_i^{s}}

\def \xrbc{{Y}^R_i}

\def \frq{{F}^{R}_i}
\def \ProjR {\Pi_i^{R_i} }

\def \xnum {{\vec{Y}}}
\def \xnumc {{Y}}
\def \xthetac{{Y}^{\theta}_i}

\def \fc {f_i}

\def  \gc{g_i}
\def \Wc{W_i}
\def \DWnc {\Delta W_{i,n}}

\def \E {\mathbb{E}}
\def \P {\mathbb{P}} 

\def \l {\left}
\def \r {\right}

\newcommand\norm[1]{\left\lVert#1\right\rVert}

\newcommand\eval[1]{\mathbb{E}\left[ #1  \right]    }

\newcommand{\der}[1]{\partial_{ #1 } } 
\newcommand{\secder}[1]{\partial ^2  _{ #1 } } 

\newcommand {\intn}[1] {\int _{t_n}^{ #1 }}
\def \sumj  {\sum_{j=1}^d}

\renewcommand{\vec}[1]{\mathbf{#1}}

\newcommand{\NZ}[1] {\mathbb{N}_{#1}}

\newtheorem{definition*}{Definition}
\newtheorem{theorem}{Theorem}

\newtheorem{lemma}{Lemma}
\newtheorem{assumption}{Assumption}
\newtheorem{corollary}{Corollary}

\newcommand{\Dt}{\Delta t}

\newcommand{\MMin}{{M_{i,n+1}}}

\newcommand{\Ymil}{\mathcal{Y}}
\newcommand{\ylbc}{{{\Ymil}^L_i}}
\newcommand{\yrbc}{{{\Ymil}^R_i}}
\newcommand{\ynum}{\vec{\Ymil}}
\newcommand{\ynumc}{{\Ymil}}

\newcommand{\rhoii}{\rho_{11}}
\newcommand{\rhoio}{\rho_{10}}
\newcommand{\rhooi}{\rho_{01}}
\newcommand{\rhooo}{\rho_{00}}

\newcommand{\EMMean}{\texttt{EM-Mean}\xspace}
\newcommand{\EMWeighted}{\texttt{EM-Weighted}\xspace}
\newcommand{\MilMean}{\texttt{Mil-Mean}\xspace}
\newcommand{\ProjEM}{\texttt{Proj-EM}\xspace}
\newcommand{\ProjMil}{\texttt{Proj-Mil}\xspace}

\newcommand{\bfzeta}{\boldsymbol{\zeta}}

\newcommand{\lte}{\text{LTE}}

\begin{document}

\title[Domain preservation]{Preserving invariant domains and strong approximation of stochastic differential equations}
\author{Utku Erdo\u{g}an}
\address{Eskisehir Technical University, Department of Mathematics,Turkey. }\email{utkuerdogan@eskisehir.edu.tr}
\author{Gabriel J. Lord}
\address{Department of Mathematics, IMAPP, Radboud University, Nijmegen, The Netherlands.}
\email{gabriel.lord@ru.nl}



\begin{abstract}
In this paper, we develop numerical methods for solving Stochastic Differential Equations (SDEs) with solutions that evolve within a hypercube $D$ in $\mathbb{R}^d$. Our approach is based on a convex combination of two numerical flows, both of which are constructed from positivity preserving methods. The strong convergence of the Euler version of the method is proven to be of order $\tfrac{1}{2}$, and numerical examples are provided to demonstrate that, in some cases, first-order convergence is observed in practice. We compare the Euler and Milstein versions of these new methods to existing domain preservation methods in the literature and observe our methods are robust, more widely applicable and that the error constant is in most cases superior.
\end{abstract}


\maketitle

\section{Introduction}
We consider the following system of SDEs in $\Rd$ given for $i=1,2,\ldots,d$ by
\begin{equation}\label{eq: componentSDE}
d\xc (t)= \fc (\vec{X}(t)) dt + G_i(\vec{X}(t))d\Wc (t), \quad \xc(0)=X_{i,0} .
\end{equation}
In particular given $L_i>R_i>0$ we take
$$G_i(\vec{y}) := g_i(\vec{y})(y_i-L_i)(R_i-y_i),$$
for any continuously differentiable function $g_i$ and $\vec{y}\in\Rd$
and impose conditions on $f_i$ and initial data $X_{i,0}$  that ensure that $X_i(t)\in (L_i,R_i)$ for all $t>0$,  (see Assumptions \ref{ass:fg} and \ref{ass:X0} below). 
The system in \eqref{eq: componentSDE} can be written as 
\begin{equation}\label{eq: vectorSDE}
d \vec{X} (t)= \vec{f} (\vec{X}(t)) dt + G(\vec{X}(t)) d\vec{W}(t),  \quad \vec{X} (0)=\vec{X}_0 
\end{equation}
where $G$ is a diagonal matrix with $G_i$ on the diagonal. 

We then have (see Theorem \ref{teo:exact_boundedness}) that $\vec{X}(t)\in D$ for all $t\in [0,T]$ where $D$ is the hypercube 
$$D:=(L_1,R_1)\times(L_2,R_2)\times \ldots \times (L_d,R_d).$$ 

The development of domain-preserving numerical methods for SDEs is an active area of research driven by the need for solutions that respect physical and practical constraints.
We are interested in the strong approximation and strong convergence of numerical methods for \eqref{eq: componentSDE}.
One well studied example is the Susceptible-Infected-Susceptible (SIS) epidemic model introduced in \cite{gray2011} that can be written as 
\begin{equation}
\label{eq:SIS}
   dI(t) = I(\eta N -\mu-\gamma - \beta I)dt + \sigma I(N-I) dW(t), 
\end{equation}
where $\beta,\mu$,$\gamma$ and $\sigma$ are parameters and $N$ is the total population size. The authors prove that $I(t)\in (0,N)$ a.s., that is $L_1=0$ and $R_1=N$.

Many methods for \eqref{eq:SIS} and other equations rely on the Lamperti transformation. One advantage of this is that methods are often of higher order and converge with rate one. A disadvantage of the Lamperti transform is that the methods are SDE specific as the transform depends on the form of the diffusion.
Neuenkirch and Szpruch  \cite{neuenkirch2014first} derived an implicit scheme based on Lamperti transformation, that preserves the domain of one-dimensional SDEs taking values in a domain such as $D$. 
Chen et. al. in \cite{chen2021first} introduced a Lamperti smoothing truncation scheme specific to \eqref{eq:SIS}. 
Yang and Huang constructed a domain preservation method for \eqref{eq:SIS} by combining the
logarithmic transformation and the Euler–Maruyama method in \cite{yang2021first} and 
Cai et. al.  \cite{cai2022}  investigated the numerical solution to a stochastic epidemic model with square-root diffusion term by modifying the classical Euler-Maruyama Method.  Another model with a square-root type diffusion is the Wright-Fisher model where solutions remain in the interval $[0,1]$ for all time almost surely. This has been examined by a number of authors, for example \cite{burrage_wright_fisher} presented a numerical method that guarantees approximations preserve the domain and \cite{STAMATIOU2018132} adapts a positivity preserving method to also obtain domain preservation for this equation. 

An alternative to the transformation approach is to either reflect or project the numerical approximation back into the domain. This is particularly appropriate for more complex domains $D$, see \cite{milsteinbook} and can be used for reflecting SDEs \cite{LeimkuhlerSharmaTretyakov}. 
Recently \cite{ulander2023boundary}  proposed an explicit boundary preserving scheme based on Lamperti transform followed by a Lie-Trotter splitting for scalar SDEs with reflection into the domain.

Note that we consider preserving a bounded domain, where  $R_i<\infty$. There have, however, been many approaches to positivity preserving SDEs. These may be based on Lamperti transform such as \cite{li2023positivity}, truncation methods, e.g. \cite{mao2021positivity}, logarithmic Euler as in \cite{lei2023}, or modification of the flow such as the balanced implicit methods \cite{Milstein19981010,Kahl2006143}.
Of direct relevance for this work is the method of \cite{bossy2021,bossy2024strongconvergenceexponentialeuler} on an exponential Euler scheme preserving positivity for the scalar case, see also \cite{brehier2023positivity} for SPDEs and more general methods based on geometric Brownian motion \cite{utkuLord}.

Our approach is applicable to a wide range of systems of general SDEs and can be extended to higher orders in a systematic manner. The method is derived by applying a positivity preserving method component wise twice to preserve $L_i$ and $R_i$. The two approximations are then combined to preserve the domain.
In Section \ref{sec:DomPresSDE} we impose our conditions of the drift and diffusion and prove the domain $D$ is preserved by solutions of the SDE. 
We illustrate how to derive the Euler version of our numerical method in \ref{sec:derivation} by first considering in Section \ref{sec:positivity} positivity preserving methods before showing in Section \ref{sec:DomPresMethod} how to construct a domain preserving method and Section \ref{sec:Milstein} illustrates the generalization to a Milstein type method.
In Section \ref{sec:StrongConv} we prove strong convergence for the Euler version of the numerical method and finally in Section \ref{sec:numerics} we compare the Euler and Milstein versions of the methods to other methods in the literature.

\section{Domain preservation for the SDE}
\label{sec:DomPresSDE}
Before we begin we define the set $\NZ{d}:=\{1,2,\ldots,d\}$ and the following useful notation for the projection of $i$th component
\begin{equation}\label{eq:Projs}  
\Projs : (y_1,\ldots,y_{i-1},y_{i},y_{i+1},\ldots, y_d) \to (y_1,\ldots,y_{i-1},s,y_{i+1},\ldots, y_d).
\end{equation}
We let $\vec{L}:=(L_1,L_2,\ldots,L_d)^T$ and $\vec{R}:=(R_1,R_2,\ldots,R_d)^T$.
We work on the probability space $(\Omega,\mathcal{F},\P)$ with normal filtration $(\mathcal{F}_t)_{t \in [0,T]}$ with $T \in (0,\infty)$. Throughout the paper, $\norm{ \cdot }$ denotes $\mathcal{\ell}_2 $ norm in $\mathbb{R}^d$. 

Extending the result given in \cite{chen2021first} for one-dimensional SDE's, we give the following existence and boundedness  result for  multi-dimensional SDEs of the form \eqref{eq: componentSDE}.
\begin{assumption}
\label{ass:fg}
    Let $\vec{f},\vec{g} \in C^1 \l(\bar{D};\Rd \r)$. For $\vec{y}\in\Rd$ and every $i\in \NZ{d}$ let 
\begin{equation} \label{eq: boundary_ineq}
\fc(\ProjL (\vec{y} ) \geq 0, \quad \text{} \quad  \fc(\ProjR (\vec{y})  ) \leq 0,
\end{equation}    
and 
\begin{equation}
G_i(\vec{y}) := g_i(\vec{y})(y_i-L_i)(R_i-y_i).    
\end{equation}
\end{assumption}
By boundedness of the continuous functions on a compact domain we define $$ K_{f_i}:=\max_{\vec{y} \in \bar{D}} \fc(\vec{y}) \quad \text{and} \quad  K_{g_i}:=\max_{\vec{y} \in \bar{D}} \gc(\vec{y}).$$ 
 Let $\partial _j \fc$ and $\partial _j g_i$  denote the partial derivative of $\fc$ and $\gc$ with respect to the $j$th variable. By Assumption \ref{ass:fg}, there exist real numbers $K_{{f}_{j,i}'}>0$ and $K_{{g}_{j,i}'}>0$  such that for $\vec{y}\in D$.
\begin{equation}\label{eq: bounded_derivative}
\l| \partial _j \fc (\vec{y}) \r| \leq K_{{f}_{j,i}'}, \quad \l| \partial _j g_i (\vec{y}) \r| \leq K_{{g}_{j,i}'}.
\end{equation}
Finally, throughout the paper, we use the constant 
\begin{equation}
\label{eq:Cconst}
C := \max_{i,j \in \NZ{d}} \{K_{f'_{j,i}},K_{g'_{j,i}},K_{f_i} ,K_{g_i} \}.
\end{equation}
\begin{assumption}
    \label{ass:X0}
    Let $\vec{X}(0)=\vec{X}_0 \in D$.
\end{assumption}
Our first result shows that the solution to \eqref{eq: vectorSDE} preserves the domain $D$. This result immediately gives bounded moments of the solution $\vec{X}(t)$ for all $t>0$. 
\begin{theorem}\label{teo:exact_boundedness} Let Assumptions \ref{ass:fg} and \ref{ass:X0} hold. Then, the SDE system \eqref{eq: componentSDE} has a unique strong solution $\vec{X}(t)$ satisfying 
for all $i \in \NZ{d}$
\begin{equation} 
\P \Big( \xc(t) \in (L_i,R_i), \quad \forall  t \geq 0 \Big)=1 .
\end{equation}
\end{theorem}
\begin{proof} We generalize the  existence proofs given in \cite{chen2021first, gray2011} for the SIS model \eqref{eq:SIS}. 
First we determine the bounds for drift components in terms of derivative bounds on the domain $D$. By the fundamental theorem of calculus and \eqref{eq: bounded_derivative}, we have
\begin{equation*}
\fc(\vec{y})-\fc(\ProjL (\vec{y}))=\int _{L_i} ^{ y_i} \partial _i \fc(\Projs (\vec{y})) 
ds \geq -K_{{f}_{i,i}'} (y_i -L_i)
\end{equation*}
and 
\begin{equation*}
\fc(\ProjR (\vec{y}) )-\fc(\vec{y})=\int _{y_i} ^{R_i} \partial _i \fc(\Projs (\vec{y})) d s \geq -K_{{f}_{i,i}'} (R_i-y_i).
\end{equation*}
Combining these inequalities and using \eqref{eq: boundary_ineq} yields
     \begin{equation} \label{eq: combined_ineq}
 -K_{{f}_{i,i}'}(y_i-L_i) \leq \fc(\vec{y}) \leq K_{{f}_{i,i}'} (R_i-y_i).
 \end{equation}
By the local Lipschitz continuity of the coefficient functions implied by Assumption  \ref{ass:fg}, we are assured the existence of a unique maximal local solution $\vec{X} (t)$ on $[0,\tau_e)$ where $\tau_e$ is the explosion time, as discussed for example in \cite{maoBook, gray2011}. 
By choosing $m_{i,0}>0$ sufficiently large such that
$$L_i+ \dfrac{1}{m_{i,0}} < \xc (0) < R_i-\dfrac{1}{m_{i,0}},$$ 
we can define a stopping time for $m_i \geq m_{i,0}$
$$
\tau_{ m_i}=\inf \Big\{ t \in [0,\tau_e) \bigg \vert \xc (t) \not\in \l( L_i+ \dfrac{1}{m_i},R_i- \dfrac{1}{m_i} \r) \Big\}.
$$
We consider the operator $\mathcal{L}$ for a sufficiently smooth scalar valued function $\phi: D \to \mathbb{R}$  
$$
\mathcal{L}\phi (\vec{y})  = \vec{f}(\vec{y})^T \nabla \phi (\vec{y})  dt   +\dfrac{1}{2} \sum _{k=1}^m \l( G_k(\vec{y}) \vec{e}_k \r)^T\nabla^2  \phi(\vec{y})  \l(G_k(\vec{y})\vec{e}_k \r)
$$
where $\vec{e}_k$ is the standard $k$th unit vector. 
We apply $\mathcal{L}$ to the  function $\mathcal{V}_i:D\to\R$  defined for any $i \in \mathbb{N}_d$ by
$$
\mathcal{V}_i(\vec{y})=\dfrac{1}{y_i-L_i}+\dfrac{1}{R_i-y_i}.
$$ 
For the solution to the SDE \eqref{eq: componentSDE}, we have
\begin{align*}
\mathcal{L}\mathcal{V}_i(\vec{X}(t))&=\l( -\dfrac{1}{(X_i(t)-L_i)^2}+\dfrac{1}{(R_i-X_i(t))^2}\r) \fc(\vec{X}(t)\\
& +{\gc(\vec{X}(t)}^2 (X_i(t)-L_i)^2 (R_i-X_i(t))^2 \l(\dfrac{1}{(X_i(t)-L_i)^3}+\dfrac{1}{(R_i-X_i(t))^3} \r).
\end{align*}
When   $t< \tau_{m_k}$ for all $k \in \NZ{d}$ 
we can obtain, using  \eqref{eq: combined_ineq},  
$$\mathcal{L} \mathcal{V}_i(\vec{X}(t)) \leq \l( K_{{f}_{i,i}'}+K_{g_i} ^2 (R_i-L_i)^2     \r) \mathcal{V}_i(\vec{X}(t)).$$ 
Let $j$ be such that 
$$\tau_{m_j}:=\min _{k \in \NZ{d}} \tau_{m_k}.$$ 
Now by applying Ito's formula 
and Gronwall's Lemma,  for $m_j>m_{j,0}$, we have
\begin{align*}\label{eq: gronwall_result}
\E \l[  \mathcal{V}_j(\vec{X}(t \wedge \tau_{ m_j}))\r] &= \mathcal{V}_j(\vec{X}(0))+\E \l[ \int_0 ^{t \wedge \tau_{ m_j}} \mathcal {L} \mathcal{V}_j(\vec{X}(s)) ds \r] \\
& \leq \mathcal{V}_j(\vec{X}(0))+\l( K_{f'_{j,j}}+K^2_{g_j}  (R_j-L_j)^2     \r) \int_0 ^{t } \E \l[  \mathcal{V}_j(\vec{X}(s \wedge \tau_{ m_j})) \r]  ds\\
& \leq \mathcal{V}_j(\vec{X}(0)) e^{ C+C ^2 (R_j-L_j)^2 } .\numberthis
\end{align*}
We now summarize the contradiction argument presented in \cite{gray2011}, to conclude that $\lim_{m_j \to \infty } \tau_{ m_j}=\infty$. 
Assume that $\lim_{m_j \to \infty } \tau_{ m_j}=\tau_{\infty}$ is finite. This implies the existence of the pair of constants  $t^*$ and $\epsilon \in (0,1)$ such that 
\begin{equation*}
\mathbb{P} \l(\tau_{\infty}\leq t^*\r) > \epsilon.
\end{equation*}
Therefore, there is an $m_{j,1}>m_{j,0}$ such that 
\begin{equation*}
\mathbb{P} \l(\Omega_{m_j} \r) \geq \epsilon, \qquad \text{for all } m_{j}>m_{j,1}
\end{equation*}
where $\Omega_{m_j}=\{ \tau_{m_j}\leq t^* \}$.  For every $\omega \in \Omega_{m_j}$, $X_j(t,\omega)$ hits the boundaries at either $L_j+\frac{1}{m_j}$ or $R_j-\frac{1}{m_j}$. By definition of $\mathcal{V}_j$,  we have 
\begin{equation*}
\mathcal{V}_j(\vec{X}(\tau_{m_j},\omega))\geq  m_j \qquad  \text{for all } \omega \in \Omega_{m_j}.
\end{equation*}
By \eqref{eq: gronwall_result}, 
\begin{align*}
\mathcal{V}_j(\vec{X}(0)) e^{ C+C^2 (R_j-L_j)^2 }&\geq \E \l[ \mathcal{V}_j(\vec{X}(t \wedge \tau_{ m_i}))\r] \\
& \geq \E \l[  \mathbbm{1} _{\Omega_{m_j}} (\omega)  \mathcal{V}_j(\vec{X}(t \wedge \tau_{ m_j}), \omega)\r]\\
& \geq m_j \E \l[ \mathbbm{1} _{\Omega_{m_j}}  \r] \geq m_j \epsilon.
\end{align*}
As $m_j$ tends to infinity, we get
$$
\mathcal{V}_j(\vec{X}(0)) e^{ C+C ^2 (R_j-L_j)^2 }= \infty
$$
which is a contradiction of the assumption $\tau_{\infty}$ was finite.  Therefore, $\lim_{m_i \to \infty }\tau_{ m_i}=\infty$ for all $i \in \NZ{d} $\end{proof}

\section{Derivation of the numerical methods}
\label{sec:derivation}
Given $N\in \mathbb{N}$ and final time $T$ we set the time step size $\Delta t=\frac{T}{N}$. This gives the uniform time partition $0=t_0<t_1    <t_2<\hdots <t_N=T$ with $t_n=n\Dt$. 

Our approach is based on taking a positivity preserving method and applying it twice in each component. Once in order to find $Y^L_i(t_n) \approx X_i(t_n)$ so that $Y^L_i(t_n)>L_i$ and once again to find $Y^R_i(t_n)\approx X_i(t_n)$ so that $Y^R_i(t_n)<R_i$ for all $i\in \NZ{d}$.
Finally a convex linear combination of  $Y^L_i(t_n)>L_i$ and $Y^R_i(t_n)<R_i$ is used to obtain our approximation $Y_i(t_n)$  to $X_i(t_n)$.
In Section \ref{sec:positivity} we introduce the positivity preserving method and in Section \ref{sec:DomPresMethod} combine these to obtain the Euler version of our method. In Section \ref{sec:Milstein} we illustrate the key steps to derive a Milstein version.

\subsection{Positivity preserving 
integrators}
\label{sec:positivity}
Consider the following SDE system with $C^1$ coefficient functions $a_i,b_i:\mathbb{R}^d \to \mathbb{R}^d$, 
\begin{equation}\label{eq:SDEsystem}
dU_i=a_i(\vec{U}(t))U_i(t)dt+b_i(\vec{U}(t))U_i(t) dW_i(t), \quad i\in \NZ{d}
\end{equation}
where $U_i(t)$ is the $i$th component of the solution $\vec{U}(t)$ to the SDE system.
We seek a numerical method of the form 
\begin{equation}
\label{eq:magnusFormComponent}
V_{i,n+1}=\exp(\mathcal{M} _{i,n}) V_{i,n}
\end{equation}
where $V_{i,n}$ is an approximation to the $U_i(t_n)$ at $t=t_n$.
Clearly from \eqref{eq:magnusFormComponent} given $V_{i,n}>0$ then by construction $V_{i,n+1}>0$.
Equivalently, we can write \eqref{eq:magnusFormComponent} as 
\begin{equation}
\label{eq:LogarithmicFormComponent}
\ln (V_{i,n+1}) -\ln (V_{i,n}) =\mathcal{M} _{i,n} .
\end{equation}
A numerical scheme of the form \eqref{eq:magnusFormComponent} can be considered as a form of exponential integrator with a diagonal exponential matrix, a geometric Brownian motion integrator \cite{utkuLord,tamedGBM,weakGBM} or as a special form of a Magnus integrator, see for example \cite{BURRAGE199934,LordMalhamWiese,DAMBROSIO2024128610}.
Consider $\ln (U_i)$, for $i\in\NZ{d}$ then by Ito's formula
$$
d\big(\ln(U_i(t))\big)=L^0\big(\ln(U_i(t)\big)dt+\sumj L^j \big(\ln(U_i(t))\big) dW_j(t)
$$
where 
$$
L^0(\cdot):=\sum_{k=0}^m a_k(\vec{U})U_k \der{k} (\cdot) + \dfrac{1}{2} \sum_{j=1}^d b_i^2 (\vec{U})U_i^2 \secder{jj} (\cdot)
$$
and
$$
L^j(\cdot):= b_j(\vec{U})U_j \der{j} (\cdot) .
$$
As a result, we have system of SDE's for $i\in\NZ{d}$
$$
\ln(U_i(t_{n+1}))=\ln(U_i(t_{n}))+ \intn{t_{n+1}} \big(a_i(\vec{U}(s)) -\dfrac{1}{2} b_i(\vec{U} (s))^2 \big) ds+ \intn{t_{n+1}}  b_i(\vec{U} (s))dW_i(s).
$$
By approximating the integrands of each integral at $t_n$
we obtain the exponential Euler scheme
\begin{equation} \label {eq:expeuler_component}
V^{EE}_{i,n+1}:=\exp(\mathcal{M}^{EE} _{i,n}) V^{EE}_{i,n}; \ \mathcal{M}^{EE} _{i,n}:=\l(a_i(\vec{V}^{EE}_n)-\dfrac{1}{2}  b_i^2(\vec{V}^{EE}_n) \r) \Dt+ b_i(\vec{V}^{EE}_n) \Delta W_{i,n}
\end{equation}
which was already derived in \cite{bossy2021} with a different approach for a scalar SDE.
To obtain higher order methods, consider the Ito-Taylor expansions of the drift 
\begin{align*} \label{eq:expansionsComponentDrift}
a_i(\vec{U}(s))-\dfrac{1}{2} b_i^2(\vec{U}(s))&=a(\vec{U}(t_n))-\dfrac{1}{2} b^2(\vec{U}(t_n))+ \intn{s} L^0 (a(\vec{U}(r))-\dfrac{1}{2}  b^2(\vec{U}(r)))dr\\
&+ \sum_{j=1}^d \intn{s} L^j (a(\vec{U}_r)-\dfrac{1}{2}  b_i^2(\vec{U}_r))dW_j (r) + h.o.t. \numberthis 
\end{align*}
and diffusion terms 
\begin{align} \label{eq:expansionsComponentDiff}
    b_i(\vec{U}(s))&=b_i(\vec{U}(t_n))+ \intn{s} L^0 (b_i(\vec{U}(r)))dr+ \sum_{j=1}^d \intn{s} L^j (b_i(\vec{U}(r))dW_j (r) + h.o.t.
\end{align}
where $h.o.t.$ stands for higher order terms.
By substituting \eqref{eq:expansionsComponentDrift} and \eqref{eq:expansionsComponentDiff} into  \eqref{eq:LogarithmicFormComponent}  repeatedly and dropping the nested integrals of higher order 
we obtain the exponential Milstein method
\begin{equation} \label {eq:expMilComponent}
V^{EMil}_{i,n+1}=\exp(\mathcal{M}^{EMil} _{i,n}) V^{EMil}_{i,n},
\end{equation}
where 
\begin{align*}
 \mathcal{M}^{EMil} _{i,n}&=\l(a_i(\vec{V}^{EMil}_{n})-\dfrac{1}{2}  b_i^2(\vec{V}^{EMil}_{n}) \r) \Dt+ b_i(\vec{Y}^{EMil}_{n}) \Delta W_{i,n}\\
 &+ \sum_{j=1}^d  b_j(\vec{V}^{EMil}_{n}) V^{EMil}_{j,n} \der{j} b_i(\vec{V}^{EMil}_{n}) \intn{t_{n+1}} \intn{s} dW_j(r) dW_i (s).
\end{align*}
This in general requires the computation of Levy areas. However, under the assumption that $b_i$ depends only on the $i$th component of the process the $\mathcal{M}^{EMil} _{i,n}$ reduces to  
\begin{align*}
 \mathcal{M}^{EMil}_{i,n}&=\l(a_i(\vec{V}^{EMil}_{n})-\dfrac{1}{2}  b_i^2(\vec{V}^{EMil}_{n}) \r) \Dt+ b_i(\vec{V}^{EMil}_{n}) \Delta W_{i,n}\\
 &+  \dfrac{1} {2} b_i(\vec{V}^{EMil}_{n}) V^{EMil}_{i,n} \der{i} b_i(\vec{V}^{EMil}_{n}) \l(\Delta W^2_{i,n} -\Dt \r)
\end{align*}
as for the classical Milstein scheme \cite{kloeden2011}.

\subsection{From positivity preservation to domain preservation}
\label{sec:DomPresMethod}
We now turn attention on how to preserve the domain $D$ and consider the SDE \eqref{eq: vectorSDE}. 

First let us examine how to preserve the boundary at $\vec{L}$ of the hypercube. Define component wise the function
\begin{equation}
    \label{eq:NewtonQuotientLEFT}
\flq(\vec{y}):=\dfrac{\fc(\vec{y}) -\fc(\ProjL(\vec{y}))}{(y_i-L_i)}, \qquad i\in\NZ{d}.
\end{equation}
We re-write the SDE \eqref{eq: componentSDE}
as the following system on the interval $[t_n,t_{n+1}]$, 
\begin{equation}\label{eq:splitODESDE}
\begin{split}
d X_i  (t)&=\fc(\ProjL (\vec{X} (t)))dt  \\
d X_i (t)&=F_i^L(\vec{X}(t)) (X_i(t)-L_i)dt+G_i(\vec{X}(t))d\Wc(t) .
\end{split}
\end{equation}
We use this to construct an approximation $Y_i^L(t)$ that preserves the left boundary at $L_i$ on the interval $t\in[t_n,t_{n+1}]$ given $\vec{Y}(t_n) =\vec{Y}_n\in D$.

For the ODE in \eqref{eq:splitODESDE} we apply the standard explicit Euler method and use this as initial data for the SDE. That is we have 
\begin{equation}
d \xlbc(t) =\flq(\xnum _n) (\xlbc(t)-L_i)dt+\gc(\xnum _n) (R_i-\xnumc_{i,n})(\xlbc(t)-L_i)d\Wc(t) ,\label{eq:left_cont_SDE}
\end{equation}
with initial data $Y_i^L(t_n) = Y_{i,n}+ \Dt f_i(\ProjL(\vec{Y}_n))$. 
We transform to $Y_i^L-L_i$ and apply the exponential Euler scheme \eqref{eq:expeuler_component} to get
\begin{equation}\label{eq: left_schema}
\xlbc (t_{n+1})=L_i+\exp\Big(\alpha^L_{i}(\xnum _n) \Dt+\beta^L_{i} (\xnum _n)\DWnc\Big)\big(\xnumc_{i,n}+\fc(\ProjL (\xnum _n)) \Dt -L_i\big)
\end{equation}
where   
\begin{equation}
\label{eq:L_variables}
\alpha^L_{i}(\vec{y})  :=\flq(\vec{y}) -\dfrac{1}{2} \big(\gc(\vec{y}) (R_i-y_{i}) \big)^2 \quad \text{and} \quad 
\beta^L_{i} (\vec{y})  :=\gc(\vec{y}) (R_i-y_{i}).
\end{equation}
By construction $\xlbc (t_{n+1})-L_i>0$ given $\vec{Y}_n\in D$ and we have preserved the boundary at $\vec{L}$.

A similar approach works for the right boundary $R_i$. Define for $i\in\NZ{d}$
\begin{equation}
    \label{eq:NewtonQuotientRIGHT}
\frq(\vec{y}):=\dfrac{\fc(\vec{y}) -\fc(\ProjR(\vec{y}) )}{(y_i-R_i)}.
\end{equation}
Rewrite the SDE as the following system on $[t_n,t_{n+1}]$
\begin{equation}\label{eq:splitODESDE_right}
\begin{split}
d X_i  (t)&=\fc(\ProjR (\vec{X} (t)))dt  \\ 
d X_i (t)&=F_i^R(\vec{X}(t)) (R_i-X_i(t))dt+G_i(\vec{X}(t))d\Wc(t) .
\end{split}
\end{equation}
we work with  the SDE 
\begin{equation}
d \xrbc(t) =\frq(\xnum _n) (R_i-\xrbc(t))dt+\gc(\xnum _n) (R_i-\xrbc(t))(\xnumc_{i,n}-L_i)d\Wc(t) ,\label{eq:right_cont_SDE}
\end{equation}
with initial data $Y_i^R(t_n) = Y_{i,n}+\Dt f_i(\ProjR(\vec{Y}_n))$  for the right boundary. 

By the transformation $R_i-\xrbc(t)$ 
we can again apply the exponential Euler scheme \eqref{eq:expeuler_component} to obtain
\begin{equation}\label{eq: right_schema}
\xrbc (t_{n+1})=R_i-\exp\big(\alpha^R_{i}(\xnum _n) \Dt+\beta^R_{i}(\xnum _n)\DWnc\big)(R_i-\xnumc_{i,n}-\fc(\ProjR(\xnum _n)) \Dt)
\end{equation}
where 
\begin{equation} \label{eq:R_variables}
\alpha^R_{i}(\vec{y}):=\frq(\vec{y}) -\dfrac{1}{2} \l(\gc(\vec{y}) (y_{i}-L_i) \r)^2 \quad \text{and} \quad \beta ^R_{i}(\vec{y}):=-\gc(\vec{y})(y_{i}-L_i).
\end{equation}

We now prove that the two approximations in \eqref{eq: left_schema} and \eqref{eq: right_schema} satisfy that $Y_i^L(t_{n+1})>L_i$ and $Y_i^R(t_{n+1})<R_i$ and then introduce our numerical method to approximate the solution of \eqref{eq: vectorSDE}.
\begin{lemma} \label{lem: impossibility} Let Assumption \ref{ass:fg} hold and
$L_i < \xnumc_{i,n} < R_i$. Suppose $Y_i^L(t_{n+1})$ is given by \eqref{eq: left_schema} and $Y_i^R(t_{n+1})$ given by \eqref{eq: right_schema}. Then for all $i\in\NZ{d}$
\begin{itemize}
\item [i)] $\forall \omega \in \Omega$,  $\xlbc (t_{n+1} ;\omega)  >L_i $ and $\xrbc (t_{n+1};\omega)   <R_i$
\item [ii)] $\lbrace \omega \in \Omega \Big \vert  \xlbc (t_{n+1} ;\omega)  >R_i \quad  \text{and} \quad \xrbc (t_{n+1} ;\omega)  <L_i   \rbrace =\emptyset $.
\end{itemize}
\end{lemma}
\begin{proof}
Due  to conditions \eqref{eq: boundary_ineq} on the boundaries, i) is obvious from the construction.  We prove ii) by contradiction. Suppose there exist a $\omega \in \Omega$ such that 
\begin{equation}\label{eq:false_assumptions}
 \xlbc (t_{n+1} ;\omega) >R_i \quad \text{and} \quad \xrbc (t_{n+1} ;\omega)  <L_i  
\end{equation}  
for $L_i < \xnumc_{i,n} (\omega)< R_i$. 
For ease of the notation, let us to skip the $\omega$ symbol. 
By  \eqref{eq: boundary_ineq} on the boundaries, the inequalities \eqref{eq:false_assumptions} implies that
\begin{align} \label{eq:intermediate_ineq}
\alpha_{i} ^L (\xnum_{n})\Dt+ \beta^L_{i} (\xnum_{n})\DWnc > - \ln \l( \dfrac{\xnumc_{i,n} -L_i}{R_i-L_i} \r) \nonumber \\
\alpha_{i} ^R(\xnum_{n}) \Dt+ \beta^R_{i} (\xnum_{n})\DWnc > - \ln \l( \dfrac{R_i-\xnumc_{i,n} }{R_i-L_i} \r)
\end{align}
where  the functions $\alpha_{i} ^L,\alpha_{i}^R,\beta_{i}^L,\beta_{i}^R$ are defined in \eqref{eq:L_variables} and \eqref{eq:R_variables}. 

By multiplying the first and second  inequalities in \eqref{eq:intermediate_ineq} by $\xnumc_{i,n}-L_i$ and $R_i-\xnumc_{i,n}$ respectively, adding together and multiplying the resulting inequality by (-1), we get the following inequality
\begin{align*}
\l(\fc(\ProjL(\xnum_{n}))  -\fc(\ProjR(\xnum_{n}))\r)\Dt&+\dfrac{1}{2} \gc (\xnum_{n})^2(\xnumc_{i,n}-L_i)(R_i-\xnumc_{i,n}) \Dt \\
& < \ln\l( \frac{(\xnumc_{i,n}-L_i) ^ {\xnumc_{i,n}-L_i} (R_i-\xnumc_{i,n}) ^ {R_i-\xnumc_{i,n}}} {(R_i-L_i)^{R_i-L_i}} \r) \\
&< \ln\l( \frac{(R_i-L_i) ^ {\xnumc_{i,n}-L_i} (R_i-L_i) ^ {R_i-\xnumc_{i,n}}} {(R_i-L_i)^{R_i-L_i}} \r)\\
& = \ln (1)=0.
\end{align*}
This gives a contradiction as the left hand side is always positive.
\end{proof}
Lemma \ref{lem: impossibility} motivates our new time stepping method.
Given an approximation $\xnum_n\in D$ to $\vec{X}(t_n)\in D$ 
then by \eqref{eq: left_schema} we find $\xnumc^L_{i}(t_{n+1})$
and by \eqref{eq: right_schema} we find $\xnumc^R_{i}(t_{n+1})$ for $i\in\NZ{d}.$
We also define the linear combination $ \xnumc^{\theta}_{i}$ of $\xnumc^L_{i}$ and $\xnumc^L_{i}$ 
\begin{equation}
\label{eq: theta update}
 \xnumc^{\theta}_{i}(t_{n+1}) :=(1-\theta_{i,n+1}) \xnumc^L _{i}(t_{n+1})+\theta_{i,n+1} \xnumc ^R_{i}(t_{n+1}), \qquad 0<\theta_{i,n+1} <1 ,
\end{equation}
$\theta_{i,n+1}$ is a parameter that may change for $n\in \NZ{N}$ and each $i\in\NZ{d}$. 

We then approximate the solution to the SDE \eqref{eq: componentSDE} $X_i(t_{n+1})$ by $\xnumc_{i,n+1}^{M_{i,n+1}}$ where 
\begin{align*}
\label{eq:method}
\xnumc_{i,n+1}^{M_{i,n+1}} = & \xnumc ^L _{i}(t_{n+1}) \mathbf{1}_{\{\xnumc ^R _{i}(t_{n+1}) \leq L_i \}} + \xnumc ^R _{i}(t_{n+1}) \mathbf{1}_{\{ \xnumc ^L _{i}(t_{n+1}) \geq R_i \}} \\ 
& +\xnumc ^{\theta} _{i}(t_{n+1}) \mathbf{1}_{\{\xnumc ^R _{i}(t_{n+1})>L_i  \text{,} \xnumc ^L _{i}(t_{n+1})<R_i \} } \numberthis.
\end{align*}
The superscript $M_{i,n+1}\in \{L,R,\theta\}$ is used to keep track of which : $\xnumc^L_{i}$,  $\xnumc^R_{i}$ or $\xnumc^\theta_{i}$ is used in \eqref{eq:method}.
We also introduce the vector notation $\vec{M}_n :=(M_{1,n},M_{2,n},\hdots,M_{d,n})$ 
such that 
$\xnum^{\vec{M}_n} (t_n) :=(\xnumc_1 ^{M_{1,n}} (t_n),\xnumc_2 ^{M_{2,n}} (t_n),\hdots, \xnumc_d ^{M_{d,n}} (t_n) )$.

%

\begin{corollary}\label{cor:boundedness_numerics}
Let Assumptions \ref{ass:fg} and \ref{ass:X0} hold. Then the numerical approximation to \eqref{eq: vectorSDE} defined in \eqref{eq:method} satisfies 
$\vec{Y}_{n}^{\vec{M}_{n}} \in D$ for all $n\in\NZ{N}$.
\end{corollary}
\begin{proof}
The result follows from Lemma \ref{lem: impossibility} and by induction.     
\end{proof}

In Section \ref{sec:StrongConv}, we prove strong convergence of the scheme. Additionally, we discuss a strategy for finding a value of $\theta_{i,n}$ 
to reduce the local truncation error at each step.

\subsection{Milstein version of the domain preserving method}
\label{sec:Milstein}
We denote  by $\Ymil_n$ the Milstein approximation to $X(t_n)$. To obtain a higher order method 
we construct Milstein versions of left and right numerical flows
by applying positivity preserving Milstein discretization \eqref {eq:expMilComponent} to the SDE's
$$
d X_i (t)=F_i^L(\vec{X}(t)) (X_i(t)-L_i)dt+G_i(\vec{X}(t))dW_i(t)
$$
and 
$$
d X_i (t)=F_i^R(\vec{X}(t)) (R_i-X_i(t))dt+G_i(\vec{X}(t))dW_i(t)
$$
on the interval $[t_n,t_{n+1}]$ with the initial conditions   $\Ymil_i^L(t_n) = \Ymil_{i,n}+\Dt f_i(\ProjL(\vec{\Ymil}_n))$ and $\Ymil_i^R(t_n) = \Ymil_{i,n}+\Dt f_i(\ProjR(\vec{\Ymil}_n))$ under the transformations $U_i=X_i-L_i$ and $U_i=R_i-X_i$. If we assume that $g_i(\vec{y})=\bar{g}_i(y_i)$, the resulting left and right flows are given by
\begin{multline*}
\ylbc(t_{n+1}) =L_i \\
  +\exp\Big(\alpha^L_{i}(\ynum _n) \Dt+\beta^L_{i} (\ynum _n)\DWnc + \gamma^ L_{i}(\ynum _n)(\DWnc ^2 -\Dt)\Big) \big(\ynumc_{i,n}+\fc(\ProjL (\ynum _n)) \Dt -L_i\big)
\end{multline*}
and 
\begin{multline*}
\yrbc (t_{n+1})=R_i \\ 
-\exp\big(\alpha^R_{i}(\ynum _n) \Dt+\beta^R_{i}(\ynum _n)\DWnc+\gamma^ R_{i}(\ynum _n)(\DWnc ^2 -\Dt)\big)(R_i-\ynumc_{i,n}-\fc(\ProjR(\ynum_n)) \Dt)
\end{multline*}
where   
\begin{align*}
\gamma^L_{i}(\vec{y}) & :=\dfrac{1}{2}\bar{g}_i(y_i)(R_i-y_i)(y_i-L_i)\Bigl( \bar{g}'_i(y_i)(R_i-y_i) -\bar{g}_i(y_i) \Bigr) \\
\gamma^R_{i}(\vec{y}) & :=-\dfrac{1}{2}\bar{g}_i(y_i)(R_i-y_i)(y_i-L_i)\Bigl( \bar{g}'_i(y_i)(y_i-L_i) +\bar{g}_i(y_i) \Bigr)
\end{align*}
and $\alpha^L_{i}$, $\beta^L_{i}$, $\alpha^R_{i}$, $\beta^R_{i}$ are same as given in \eqref{eq:L_variables},\eqref{eq:R_variables}. Same as in \eqref{eq: theta update}, the theta flow is given by
\begin{equation}
\label{eq: theta update_mil}
 \Ymil^{\theta}_{i,n+1} :=(1-\theta_{i,n+1}) \Ymil^L _{i,n+1}+\theta_{i,n+1} \Ymil ^R_{i,n+1}        , \qquad 0<\theta_{i,n+1} <1 .
\end{equation}
This results in the following method where $\theta_{i,{n+1}}$ is yet to be determined.
\begin{align*}
\label{eq:method_milstein}
\Ymil_{i,n+1}^{M_{i,n+1}} = & \Ymil ^L _{i,n+1} \mathbf{1}_{\{\Ymil ^R _{i} \leq L_i \}} + \Ymil ^R _{i} \mathbf{1}_{\{ \Ymil ^L _{i,n} \geq R_i \}} \\ 
& +\Ymil ^{\theta} _{i,n} \mathbf{1}_{\{\Ymil ^R _{i,n+1}>L_i  \text{,} \Ymil ^L _{i,n+1}<R_i \} } \numberthis.
\end{align*}

\section{Convergence analysis of Euler method and parameter selection}
\label{sec:StrongConv}
For our convergence analysis, we define the pathwise error for the $ith$ component at time level $n$ by 
\begin{equation} \label{eq: error_definition_component}
E_{i,n}:=X_i (t_n)-\xnumc_{i,n}^{M_{i,n}}
\end{equation}
with $\MMin \in \{L,R,\theta \}$ 
and define corresponding vectors 
\begin{equation} \label{eq: error_definition_vector}
\vec{E}_{n}:=\vec{X}(t_n)-\xnum_n^{\vec{M}_n}.
\end{equation}

\subsection{Ito-Taylor expansions}
We begin with giving the Ito-Taylor expansions of numerical the approximation over the interval $[t_n,t_{n+1}]$ in \eqref{eq:method} and then give the expansion for the solution to \eqref{eq: vectorSDE}. For \eqref{eq:method} we give the expansions for ${Y}_i^L(t)$,  ${Y}_i^R(t)$ and  then ${Y}_i^\theta(t)$. 
We assume that at time $t_n$ $\vec{X}(t_n)\in D$ and we are given $\bfzeta \in D$ so that $\bfzeta \approx \vec{X}(t_n)$ and in our error analysis below we have $\bfzeta=\xnum^{\mathbf{M}_{n}}$. 
The expansions can then be used to obtain one step errors when $\MMin=L$ and the update from \eqref{eq: left_schema} is used in \eqref{eq:method}; when  $\MMin=R$ (\eqref{eq: right_schema} is used in \eqref{eq:method}) and when $\MMin=\theta$ so \eqref{eq: theta update} is used in \eqref{eq:method}.

All our Ito-Taylor expansions for the numerical method have remainder terms of the form
\begin{align}
\label{eq:remainder}
    r_i^{M_{i,n+1}}(t_n,t_{n+1})=& r_{11}^\MMin + r_{10}^\MMin
    + r_{01}^\MMin + r_{00}^\MMin,
\end{align}
where the four terms  $r_{11}^\MMin$, $r_{10}^\MMin$, $r_{01}^\MMin$ and $r_{00}^\MMin$ are given by 
\begin{align*}
r_{11}^\MMin :=& \intn{t_{n+1}} \intn{s} I_{11}^\MMin dW_i(\tau) dW_i(s) ; \\
r_{10}^\MMin := & \intn{t_{n+1}} \intn{s} I_{10}^\MMin dW_i(\tau) ds; \\
r_{01}^\MMin :=& \intn{t_{n+1}} \intn{s} I_{01}^\MMin d\tau dW_i(s) +\intn{t_{n+1}} \intn{s} J_{01}^{M_{i,n+1}} d\tau dW_i(s);\\
r_{00}^\MMin := &\intn{t_{n+1}} \intn{s} I_{00}^\MMin d\tau ds
+ \intn{t_{n+1}} \intn{s}  J_{00}^{M_{i,n+1}} d\tau ds;
\end{align*}
and $I_{11}^\MMin, I_{10}^\MMin, I_{01}^\MMin, I_{00}^\MMin$ and $J_{01}^\MMin$ and $J_{00}^\MMin$ are defined below for the different cases of $\MMin$.

We start by considering when $\MMin=L$.
Consider the integral representation of \eqref{eq:left_cont_SDE} for $t\in[t_n,t_{n+1}]$ 
\begin{align*}\label{eq:leftflowIto2} 
 \xlbc(t)=&\zeta_{i}+\fc(\ProjL (\bfzeta))\Dt 
 + \intn{t} \flq(\bfzeta) \l(   \xlbc(s)-L_i \r) ds \\
 &+ \intn{t} g_i (\bfzeta) \l(R_i-\zeta_{i} \r) \l(  \xlbc(s)-L_i \r)
 dW_i (s). \numberthis
\end{align*}
Evaluating \eqref{eq:leftflowIto2} at $t_{n+1}$ and substituting in the same expansion for $\xlbc(s)$ in \eqref{eq:leftflowIto2} we get
\begin{align*} \label{eq:left_expansion}
  \xlbc(t_{n+1})=&\zeta_{i}+  \intn{t_{n+1}} f_i(\bfzeta)  ds + 
 \intn{t_{n+1}}  G_i(\bfzeta)dW_i (s) +r^L_i (t_n,t_{n+1}) \numberthis
\end{align*}
where $r^L_i$ is defined as in \eqref{eq:remainder} with 
\begin{equation}\label{eq:remainder_left}
\begin{cases}
I_{11}^L & := \l( g_i (\bfzeta) \l(R_i-\zeta_{i} \r) \r)^2  \l(   \xlbc(\tau)-L_i \r) \\
I_{10}^L & :=\flq(\bfzeta) g_i (\bfzeta) \l(R_i-\zeta_{i} \r)   \l(   \xlbc(\tau)-L_i \r) \\
I_{01}^L & := \flq(\bfzeta) g_i (\bfzeta) \l(R_i-\zeta_{i} \r)   \l(   \xlbc(\tau)-L_i \r) \\
J_{01}^L & := g_i(\bfzeta)(R_i-\zeta_{i})\fc(\ProjL (\bfzeta)) \\
I_{00}^L & := \flq(\bfzeta) ^2  \l(   \xlbc(\tau)-L_i \r) \\
J_{00}^L & := \flq(\bfzeta) \fc(\ProjL (\bfzeta)).
\end{cases}
\end{equation}
When $M_{i,n+1}=R$, we have from a similar analysis that  
\begin{align*}\label{eq:right_expansion}
 \xrbc(t_{n+1}) = & \zeta_{i}+ \intn{t_{n+1}} f_i(\bfzeta)  ds + 
 \intn{t_{n+1}} G_i(\bfzeta) dW_i(s) +r^R_i (t_n,t_{n+1}) \numberthis
\end{align*}
where  
\begin{equation}\label{eq:remainder_right}
\begin{cases}
I_{1,1}^R & :=   \l( g_i (\bfzeta) \l(\zeta_{i}-L_i \r) \r)^2  \l(  R_i- \xrbc( \tau) \r)\\
I_{1,0}^R & := \frq(\bfzeta) g_i (\bfzeta) \l(\zeta_{i}-L_i \r)   \l( R_i-  \xrbc( \tau)  \r) \\
I_{0,1}^R & := \frq(\bfzeta) g_i (\bfzeta) \l(\zeta_{i}-L_i \r)   \l( R_i - \xrbc( \tau) \r) \\
J_{0,1}^R & :=  g_i(\bfzeta)(\zeta_{i}-L_i) \fc(\ProjR (\bfzeta))\\
I_{0,0}^R & :=  \frq(\bfzeta) ^2  \l( R_i - \xrbc(\tau) \r) \\
J_{0,0}^R & := \frq(\bfzeta) \fc(\ProjR (\bfzeta)).
\end{cases}
\end{equation}
For the linear combination, $M_{i,n+1}=\theta$, we have 
\begin{align} \label{eq: theta_expansion}
 \xthetac(t_{n+1})=\zeta_{i}+  \intn{t_{n+1}} f_i(\bfzeta)  ds +
 \intn{t_{n+1}} G_i(\bfzeta)dW_i (s)
 +r^{\theta}_i (t_n,t_{n+1})
\end{align}
and the remainder term is the linear combination 
$$
r^{\theta}_i(t_n,t_{n+1}) := (1-\theta_{i,n+1}) r^L_i(t_n,t_{n+1})+ \theta_{i,n+1} r^R_i(t_n,t_{n+1}).
$$
Finally we provide the Ito-Taylor expansion for the solution to the domain preservation SDE \eqref{eq: vectorSDE} on $[t_n,t_{n+1}]$
\begin{align*} \label{eq:exact_expansion}
X_i(t_{n+1} )= & X_i(t_n)+ \intn{t_{n+1}} f_i (\vec{X}(t_n))ds  + \intn{t_{n+1}}  
G_i(\vec{X}(t_n)) dW_i(s) + r^{X}_i(t_n,t_{n+1}) \numberthis
\end{align*}
where $r_i^{X} := r_{11}^{X} + r_{10}^{X}
    + r_{01}^{X} + r_{00}^{X},$ with
\begin{align*}
r^{X}_{11}&=\sumj \intn{t_{n+1}} \intn{s}  I^{X}_{11,j} dW_j (\tau) dW_i(s);\qquad  &
r^{X}_{10}=\sumj \intn{t_{n+1}} \intn{s}   I^{X}_{10,j} dW_j (\tau)  ds;\\
r^{X}_{01}&=\sumj \intn{t_{n+1}} \intn{s} I^{X}_{01,j} d\tau dW_i (s);  & 
r^{X}_{00}=\sumj \intn{t_{n+1}} \intn{s} I^{X}_{00,j} d\tau ds;
\end{align*}
and 
\begin{align*}\label{eq:remainder_exact}
I^{X}_{11,j} & :=  G_j(\vec{X}(\tau)) \der{j}\l(G_i (\vec{X}(\tau) \r)  \\
I^{X}_{10,j} & := G_j(\vec{X}(\tau)) \der{j} \l(f_i (\vec{X}(\tau)) \r)  \\
I^{X}_{01,j} & :=  \l( f_j(\vec{X}(\tau))  \der{j}G_i(\vec{X}(\tau) )  + \dfrac{1}{2} G^2_j (\vec{X}(\tau))  \secder{j} \l(G_i(\vec{X}(\tau) ) \r)   \r)  \\
I^{X}_{00,j} & := \l( f_j(\vec{X}(\tau))  \der{j}f_i(\vec{X}(\tau))  + \dfrac{1}{2} (G^2_j(\vec{X}(\tau))  \secder{j} \l(f_i(\vec{X}(\tau)) \r)   \r). 
\end{align*}


\subsection{Strong convergence}
We use the Ito-Taylor expansions to prove a strong convergence result for the Euler domain preservation method \eqref{eq:method}.

\begin{theorem}
Let Assumptions \ref{ass:fg} and \ref{ass:X0} hold.
Then there exists a constant $K>0$ such that for all $n \in\NZ{N}$
$$ 
\eval{\norm{\vec{E}_{n}}^2} \leq K \Dt e^{KT}.
$$
\end{theorem}
\begin{proof}
By \eqref{eq: error_definition_component} and the Ito-Taylor expansions given in \eqref{eq:left_expansion}, \eqref{eq:right_expansion},\eqref{eq: theta_expansion} and \eqref{eq:exact_expansion} we have 
\begin{align*} 
E_{i,n+1} & =  X_i (t_{n+1})-\xnumc_{i,n+1}^{M_{i,n+1}} \\
& = E_{i,n}+  \intn{t_{n+1}} f_i (\vec{X}(t_n))- f_i (\xnum_n^{M_n}) ds\\
& \quad + \intn{t_{n+1}} \left( G_i(\vec{X}(t_n)) - G_i(\xnum_n^{M_n}) \right) dW_i (s)
 + r^{X}_i(t_n,t_{n+1})-r^{M_{i,n+1}} _i(t_n,t_{n+1}).
\end{align*}
Introducing shorthand notation $A^s_n$ for the deterministic integral and $B_n^w$ for the stochastic integral we write this as 
\begin{equation} \label{eq: single_err_component}
   E_{i,n+1}= E_{i,n} + A_n^s + B_n^w +  \rhoii+\rhoio+\rhooi+\rhooo
\end{equation}
where 
\begin{align*}
&
\rhoii=\sumj\intn{t_{n+1}} \intn{s}  (I^{X}_{11,j}-I^{M_{i,n+1}}_{11}) dW_j (\tau) dW_i(s) - \intn{t_{n+1}} \intn{s} I_{11}^\MMin dW_i(\tau) dW_i(s)\\
&
\rhoio:=\sumj \intn{t_{n+1}} \intn{s}   I^{X}_{10,j}dW_j (\tau)  ds-\intn{t_{n+1}} \intn{s}  I^{M_{i,n+1}}_{10}  dW_j (\tau)  ds \\
&
\rhooi:=\sumj \intn{t_{n+1}} \intn{s} I^{X}_{01,j}d\tau dW_i (s)-\intn{t_{n+1}} \intn{s}  (I^{M_{i,n+1}}_{01}+J^{M_{i,n+1}}_{01} ) d\tau dW_i (s)\\
&
\rhooo:=\sumj \intn{t_{n+1}} \intn{s} 
I^{X}_{00,j}d\tau ds-\intn{t_{n+1}} \intn{s} ( I^{M_{i,n+1}}_{00}  +J^{M_{i,n+1}}_{00})d\tau ds.
\end{align*}
By the order of the nested integrals as given in  \cite[{Lemma 2.2}]{milsteinbook}, and using Assumption \ref{ass:fg}, Theorem \ref{teo:exact_boundedness} and Corollary \ref{cor:boundedness_numerics} 
\begin{equation}\label{eq:constants}
\eval{\rhooo^2} \leq \tilde{C}_{00} \Dt ^4; \quad 
\eval{\rhoio^2} \leq \tilde{C}_{10} \Dt ^3;\quad
\eval{\rhooi^2} \leq \tilde{C}_{01} \Dt ^3; \quad
\eval{\rhooo^2} \leq \tilde{C}_{11} \Dt ^2.
\end{equation}
By taking square of \eqref{eq: single_err_component}, we have
\begin{align*}
 E_{i,n+1}  ^2=& E_{i,n}  ^2+2  E_{i,n}A_n^s +2  E_{i,n}B_n^w+2E_{i,n}(\rhoii+\rhoio+\rhooi)+ 2 E_{i,n}\rhooo  \\
 & +A_n^s{}  ^2+ 2   A_n^s B_n^w+2 A_n^s (\rhoii+\rhoio+\rhooi)+ 2  A_n^s \rhooo\\
 & +B_n^w{}  ^2+2 B_n^w (\rhoii+\rhoio+\rhooi )+ 2  B_n^w \rhooo\\
 &+ (\rhoii+\rhoio+\rhooi)^2+2(\rhoii+\rhoio+\rhooi ) \rhooo  + \rhooo ^2.
\end{align*}
When taking expectation, the following terms have zero expectation 
\begin{align*}
\eval{E_{i,n}B_n^w}=0, \qquad & \eval{E_{i,n}(\rhoii+\rhoio+\rhooi)}=0\\
\eval{A_n^s B_n^w}=0, \qquad & \eval{A_n^s (\rhoii+\rhoii+\rhooi )}=0\\
\eval{B_n^w \rhooo}=0, \qquad & \eval{ \l (\rhoii+\rhoio+\rhooi\r) \rhooo}=0.\\
\end{align*}
By Assumption \ref{ass:fg}, we have  
\begin{align*} \label{eq: A_bounds}
\eval{E_{i,n}A_n^s}&= \Dt \eval{E_{i,n} \left(  f_i (\vec{X}(t_n))- f_i (\xnum_n^{M_n}) \right)} \leq C \Dt \eval{\norm{\vec{E}_{n}}^2}\\
\eval{A_n^s{}  ^2}= & \Dt^2 \eval{ \left(  f_i (\vec{X}(t_n))- f_i (\xnum_n^{M_n} ) \right)^2 } \leq C^2 \Dt^2  \eval{\norm{\vec{E}_{n}}^2}. \numberthis
\end{align*}
By the continuously differentiability assumption on  the functions $g_i$ and Ito's isometry, we have 
%
\begin{equation}\label{eq: G_bound}
\eval{B_n^w{}  ^2}  \leq \tilde{C}_G^2  \Dt \eval{\norm{\vec{E}_{n}}^2}.
\end{equation}
where, for $C$ from \eqref{eq:Cconst} 
$$\tilde{C}_G= \frac{C}{4}\max_{i\in\NZ{d}}(R_i-L_i)^2+2C\max_{i\in\NZ{d}}(R_i-L_i).$$
By Jensen's inequality for sums, and using \eqref{eq:constants} we get 
\begin{equation} \label{eq: 3r_bound}
\eval{\l(\rhoii+ \rhoio+\rhooi\r)^2}  \leq 3 (\tilde{C}_{11}+\tilde{C}_{10}+\tilde{C}_{01}) \Dt ^2.
\end{equation}
Applications of the stochastic Cauchy-Schwarz inequality and the arithmetic-geometric mean inequality gives the following three bounds
\begin{align*}
\eval{E_{i,n} \rhooo } & \leq \eval{\norm{\vec{E}_{n}}^2} ^{1/2} \eval{\l(\rhooo \r)^2} ^{1/2}   \leq  \eval{\norm{\vec{E}_{n}}^2} \sqrt{\tilde{C}_{00}} \Dt^{1/2} \Dt^{3/2} \\
& \leq \dfrac{ \sqrt{\tilde{C}_{00}} }{2} \left(\eval{\norm{\vec{E}_{n}}^2} \Dt + \Dt^3 \right). \numberthis
\end{align*}

\begin{align*}
\eval{A_n^s \rhooo}& \leq \eval{A_n^s {} ^2} ^{1/2} \eval{\l( \rhooo \r)^2} ^{1/2}  \leq C \sqrt{\tilde{C}_{00}} \Dt ^{1/2}  \eval{\norm{\vec{E}_{n}}^2} ^{1/2}  \Dt^{5/2} \\
& \leq \dfrac{C  \sqrt{\tilde{C}_{00}} }{2} \left(\eval{\norm{\vec{E}_{n}}^2} \Dt + \Dt^5 \right) \numberthis
\end{align*}

\begin{align*}
\eval{B_n^w \l(\rhoii+ \rhoio+\rhooi\r) }& \leq \eval{B_n^w {} ^2} ^{1/2} \eval{\l(\rhoii + \rhoio+\rhooi\r)^2} ^{1/2} \\
& \leq \tilde{C}_G  \Dt^{1/2}  \eval{\norm{\vec{E}_{n}}^2} ^{1/2} \sqrt{3} \sqrt{\tilde{C}_{10}+\tilde{C}_{01}+\tilde{C}_{11}} \Dt \\
& \leq \dfrac{ \tilde{C}_G}{2}   \sqrt{3} \sqrt{\tilde{C}_{10}+\tilde{C}_{01}+\tilde{C}_{11}}  \left(\eval{\norm{\vec{E}_{n}}^2} \Dt + \Dt^2 \right).
\end{align*}
Summing over time, the component wise  error is given by
 \begin{equation*}
\eval{\l| E_{i,n} \r|^2}  \leq \lambda  \Dt + \lambda  \Dt \sum_{k=0} ^{n-1} \eval{\norm{\vec{E}_{k}}^2}.
 \end{equation*}
 where 
 \begin{multline*}
 \lambda:= \\ \max \Big\{ 2C,  \sqrt{\tilde{C}_{00}} ,C^2, C\sqrt{\tilde{C}_{00}} ,  \tilde{C}_G^2, \tilde{C}_G  \sqrt{3} \sqrt{\tilde{C}_{10}+\tilde{C}_{01}+\tilde{C}_{11}},3 (\tilde{C}_{10}+\tilde{C}_{01}+\tilde{C}_{11}),\tilde{C}_{00} \Big\}.
 \end{multline*}
By summation over the $d$ components and the discrete Gronwall Lemma, we have
 \begin{align*}
\eval{\norm{\vec{E}_{n}}^2} & \leq \lambda  d \Dt + \lambda d \Dt \sum_{k=0} ^{n-1} \eval{\norm{\vec{E}_{k}}^2} \\
\eval{\norm{\vec{E}_{n}}^2}& \leq \lambda d \Dt e^{K T}.
 \end{align*}
 Hence the result is proved.
\end{proof}

\subsection{Selection of $\theta$ and local truncation error for the Euler method}
When $M_{i,n+1}=\theta$, we have to choose a value for $\theta_{i,n+1}$.
In the deterministic setting for a $\theta$ method the choice of $\theta=\tfrac{1}{2}$ is determined from a local error analysis.
Clearly we could choose $\theta$ to be a constant.
In Section \ref{sec:numerics} we take $\theta_{i,n}=\tfrac{1}{2}$ for all $i\in\NZ{d}$ and $n\in\NZ{N}$ and denote the method with this choice \EMMean.

Our aim now is to pick $\theta_{i,n+1}$ to reduce the local error, we will denote this method \EMWeighted.
For a local error analysis, we need to work with numerical and exact flows but having same initial condition at $t=t_n$, say $\xnum_n \in \mathbb{R}^d$. 

For simplicity let us assume that 
$\fc(\ProjL(\xnum_{n}))= \fc(\ProjR(\xnum_{n})) =0$. and so examine the error from the approximation of the diffusion.
We then define
\begin{align*} \label{eq:theta_flow}
\Psi ^L (t_n,t_{n+1},\xnum_n)& := L_i+e^{\alpha^L_{i}(\xnum_n ) \Dt+\beta^L _{i} (\xnum_n)\DWnc}(\xnumc_{i,n}-L_i) \nonumber \\
\Psi ^R (t_n,t_{n+1},\xnum_n)& :=R_i-e^{\alpha^R_{i}(\xnum_n) \Dt+\beta^R _{i}(\xnum_n)\DWnc}(R_i-\xnumc_{i,n})  \nonumber \\
\Psi ^{\theta} (t_n,t_{n+1},\xnum_n)& :=(1-\theta_{i,n+1}) \Psi ^{L} (t_n,t_{n+1},\xnum_n)+\theta_{i,n+1} \Psi ^{R} (t_n,t_{n+1},\xnum_n).
\end{align*}
Furthermore we denote the flow corresponding to the $i$th component of \eqref{eq: componentSDE} starting from $\xnum_n$ by $\Psi^{X} (t_n,t_{n+1},\xnum_n)$. 
Then the one step error for the $i$th  component is expressed as
\begin{equation*} 
\lte_i:=\Psi ^{\theta} (t_n,t_{n+1},\xnum_n)-\Psi ^{X} (t_n,t_{n+1},\xnum_n).
\end{equation*}
By recalculating the Ito-Taylor expansions for numerical and exact flows as in \eqref{eq:left_expansion} \eqref{eq:right_expansion}, \eqref{eq: theta_expansion}, 
\eqref{eq:exact_expansion}, we have 
\begin{align*} 
\lte_i=&(1-\theta_{i,n+1}) \intn{t_{n+1}} \intn{s} \l( g_i (\xnum_n) \l(R_i-\xnumc_{i,n} \r) \r)^2  \l(   \xnumc_{i,n}-L \r) dW_i(\tau) dW_i(s) \nonumber \\
&-\theta_{i,n+1} \intn{t_{n+1}} \intn{s} \l( g_i (\xnum_n) \l(\xnumc_{i,n}-L_i \r) \r)^2  \l(  R_i- \xnumc_{i,n} \r) dW_i(\tau) dW_i(s) \nonumber \\
& - \sumj \intn{t_{n+1}}  \intn{s}  g_j (\xnum_n) (\xnumc_{j,n} -L_j)  (R_j-\xnumc _{j,n}) \der{j}G_i (\xnum_n)dW_j (\tau) dW_i(s) + h.o.t
\end{align*}
where $h.o.t.$ refers to higher order terms obtained by repeated substitution of Ito representation of each flow. 

Under the assumption that the noise is diagonal, i.e. $g_i(\xnum_n)=\tilde{g}_i (\xnumc_{i,n})$, this reduces to 
\begin{align*} \label{eq: lte}
\lte_i=& \Xi_n  \tilde{g}_i (\xnumc_{i,n}) \l(R_i-\xnumc_{i,n} \r) \l(   \xnumc_{i,n}-L_i \r)  \intn{t_{n+1}} \intn{s}dW_i(\tau) dW_i(s) +h.o.t.
\end{align*}
where 
\begin{multline*}
\Xi_n:=\l(1-\theta_{i,n+1}\r) \tilde{g}_i (\xnumc_{i,n}) \l(R_i-\xnumc_{i,n} \r)   -  \theta_{i,n+1}   \tilde{g}_i (\xnumc_{i,n})    \l(  \xnumc_{i,n} -L_i \r) \\ - \der{i} \Bigl(\tilde{g}_i (\xnumc_{i,n}) (\xnumc_{i,n}-L_i)  (R_i-\xnumc_{i,n}) \Bigr) .
\end{multline*}
For the first term above to vanish $\theta_{i,n+1}$ should be selected such that
$\Xi_n=0$. 
Therefore, taking
\begin{equation}\label{eq: optimal_theta}
\theta_{i,n+1}=\frac{\xnumc_{i,n}-L_i}{R_i-L_i}\l(1-\frac{\der{i}\tilde{g} (\xnumc_{i,n}) }{\tilde{g} (\xnumc_{i,n})} \l(R_i-\xnumc_{i,n}\r)\r)
\end{equation}
eliminates the leading term and the local error is reduced to $\eval{\lte_i^2 }\leq K \Dt^3$. 

We see below in Section \ref{sec:numerics} that this choice can improve the observed rate of convergence of the method.

\section{Numerical Examples}
\label{sec:numerics}
We compare our numerical methods to other methods that also preserve the domain.
We consider two versions of the Euler method 
\eqref{eq: theta update}.
We denote by \EMMean the method in \eqref{eq: theta update} with $\theta_{i,n}=\tfrac{1}{2}$ for all $i$ and $n$.
We denote the Euler method with $\theta_{i,n}$ from \eqref{eq: optimal_theta} by \EMWeighted. 
We also consider the Mistein version of the scheme in \eqref{eq: theta update_mil} with a fixed value of $\theta=\tfrac{1}{2}$
and denote this \MilMean.
%
%
We introduce projected versions of the standard Euler (denoted \ProjEM) and Milstein method (denoted \ProjMil). These are constructed by considering
$$X_{i,n+1}= \max(\min(R_i,\tilde{X}_{i}),L_i)$$
where $\tilde{X}$ is either the standard Euler 
$$\tilde{X_i}=X_{i,n}+f_i(\vec{X}_n)\Dt +G_i(\vec{X}_n)\Delta W_{i,n} $$
or standard Milstein update
$$\tilde{X_i}=X_{i,n}+f_i(\vec{X}_n)\Dt +G_i(\vec{X}_n)\Delta W_{i,n} + G_i(\vec{X}_n)G_i'(\vec{X}_n)(\Delta W_{i,n}^2-\Dt).$$ 
Although these are natural methods to consider we are not aware of strong convergence proofs for these methods.
We now consider four different examples and, where other domain preservation methods have been developed that may be applied, we compare the methods above to those.

In our implementations of \EMMean, \EMWeighted and \MilMean for numerical stability it can be beneficial to add an subtract the 
the drift values at the boundaries in \eqref{eq:splitODESDE} \eqref{eq:splitODESDE_right}. This is useful  when the drift term takes very small values in magnitude and was proposed in \cite{bossy2021}. 
This was applied in example (3) for parameter set (i). Although the method converges without applying this technique the error constant is larger. 



\subsection*{Example 1: An SDE with exact  solution.} 
The one-dimensional SDE
\begin{equation} \label{eq: SDEwithExact}
dX(t) = -\beta^2 X(t)(1-X^2(t)) dt + \beta (1-X^2(t)) dW(t)
\end{equation}
is introduced in \cite{kloeden2011} with the exact solution $X(t)\in (-1,1)$ a.s. given by
$$
X(t) = \frac{(1+X_0)e^{2\beta W(t)} + X_0 -1 }{(1+X_0)e^{2\beta W(t)} + 1 - X_0}.
$$
We run the numerical experiment on the interval $t \in [0,4]$ for  $\beta=2$ and $X_0=0.9$ and compare the Artificial Barrier Euler-Maruyama Method (\texttt{ABEM}) \cite{ulander2024artificial}, to \EMMean, \EMWeighted, \ProjEM and \ProjMil.
Sample paths produced by these schemes for a single realization is shown in Figure \ref{fig: withExact} (a) with a uniform time step $\Dt=2^{-7}$. We clearly see all the methods preserve the domain and the methods analysed here remain close to the solution.
Figure \ref{fig: withExact} (b) shows the root mean square error estimated with 2000 realizations. Note that although the theoretical rate of convergence for \EMWeighted is $\tfrac{1}{2}$ we observe a rate of convergence closer to one. Furthermore we observe that the error constant for \EMWeighted is the smallest.
\begin{figure}
\begin{center}
(a) \hspace{0.49\textwidth} (b)\\  \includegraphics[width=0.49\textwidth,height=0.25\textheight]{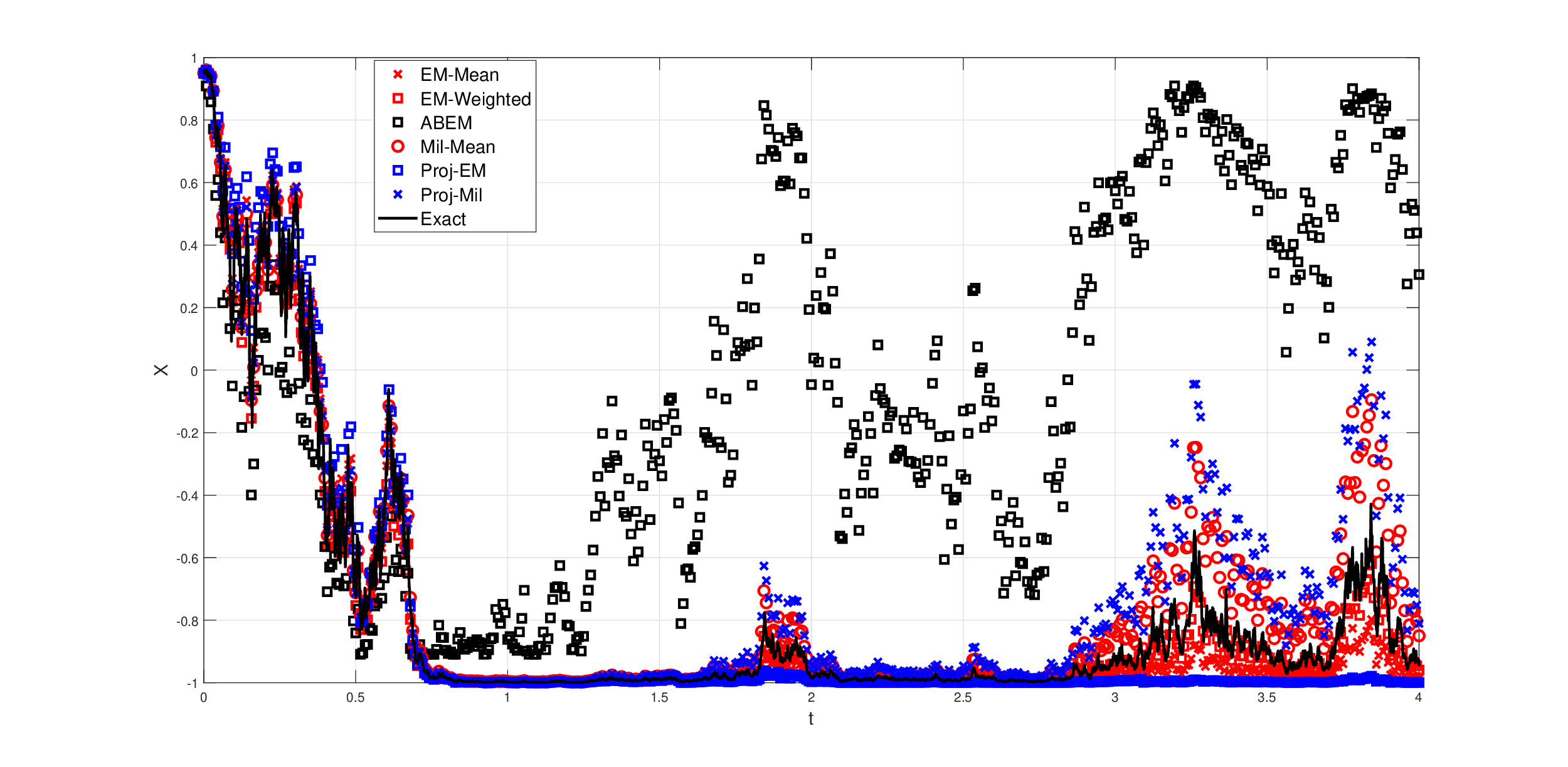}
\includegraphics[width=0.49\textwidth,height=0.25\textheight]{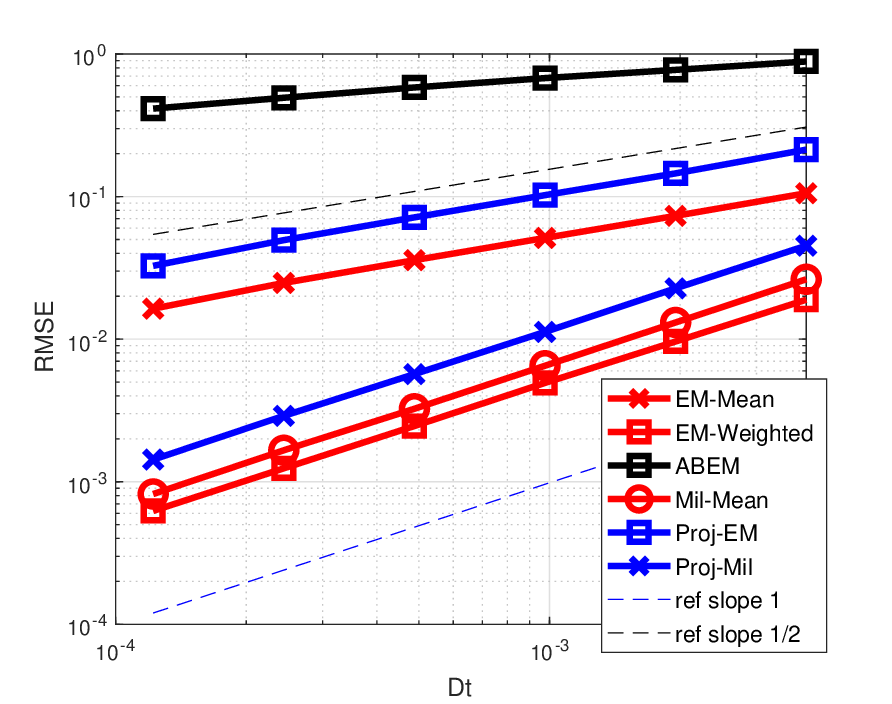}
\end{center}
\caption{Sample paths for example 1 (a) and in (b) RMSE against $\Delta t$.}
  \label{fig: withExact}
\end{figure}

\subsection*{Example 2:  Trigonometric diffusion term}
Consider the one-dimensional SDE  
\begin{equation} \label{eq: SDEwithTrig}
dX(t)=X(t) (1-X(t))dt+\sin(\pi X(t))dW(t).
\end{equation}
Note that the diffusion coefficient could be restated as  $g(x) x (1-x)$ for  $g(x)=\frac{\sin(\pi x)}{x (1-x)}$ and that both $g$ and $g'$ are bounded on (0,1) with   finite limits at $0$ and $1$. Therefore our methods are applicable for \eqref{eq: SDEwithTrig}.  We estimate the root mean square error with 2000 samples with two different settings for the initial data for $t\in[0,1]$. In the first experiment, initial values for each realizations are randomly chosen from $U(0,1)$. In the second example,  the initial value is fixed $X(0)=0.95$ for all realizations. Since we do not have an exact solution in this case reference solutions are obtained using the method \MilMean with $\Delta t=2^{-18}$. Furthermore we are not aware of another method for domain preservation for this type of example.
In Figure (\ref{fig: withTrigRandom}) (a) orders of strong errors are compared for the random initial data and in (b) for the fixed initial data.
We observe that \MilMean seems to have the smallest error constant in this example and once again \EMWeighted shows a much improved rate of convergence of the theoretical rate and is competitive with the Milstein method. 

\begin{figure}
\begin{center}
    (a) \hspace{0.48\textwidth} (b)\\
  \includegraphics[width=0.48\textwidth,height=0.25\textheight]{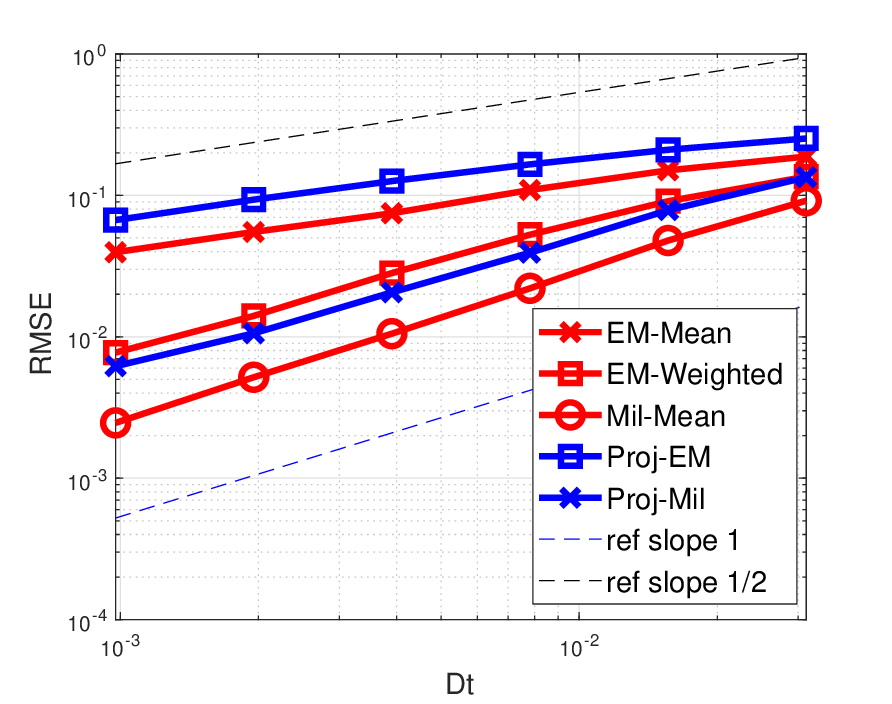}
    \includegraphics[width=0.48\textwidth,height=0.25\textheight]{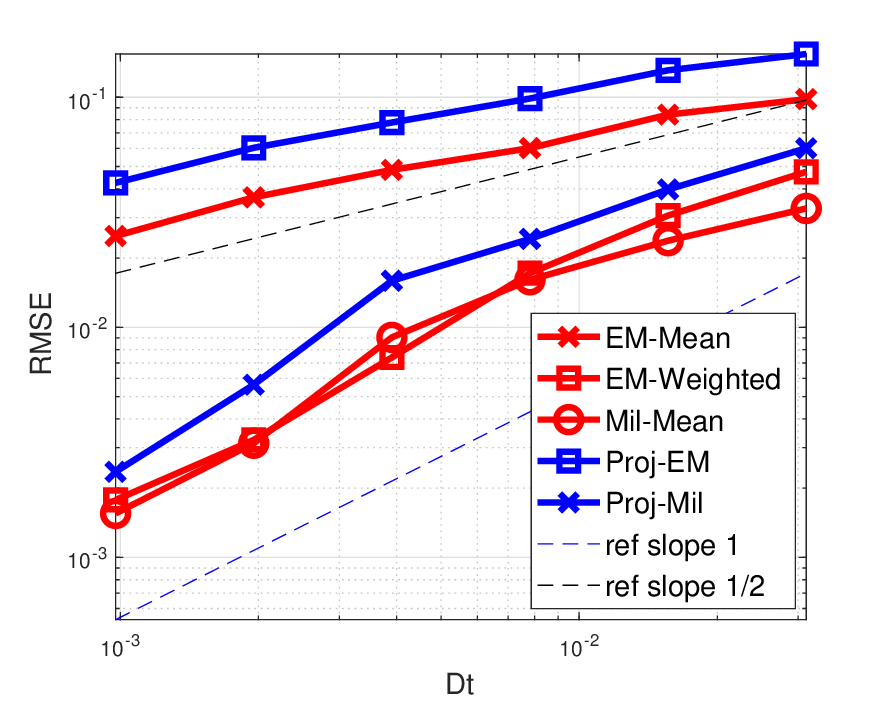}
  \caption{RMSE against $\Delta t$  for example 2 with  (a) random initial values drawn from $U(0,1)$ and (b) for fixed initial data $X(0)=0.95$.}
  \label{fig: withTrigRandom}
\end{center}
\end{figure}

\subsection*{Example 3:  Susceptible–Infected–Susceptible (SIS)}
Reconsider the SIS model
\begin{equation*}
   dI = [\eta I(t) -\beta I(t)^2] dt + \sigma [N-I(t)]I(t) dW(t)
\end{equation*}
 derived and analysed in \cite{gray2011} and investigated numerically by a number of authors for different parameter values. 
We consider two sets of parameters, both with $t\in[0,4]$ and take 2000 realizations. For the first of these (Figure \ref{fig: sis_mid} (a)) we compare to methods based on the Lamperti tranformation for this equation
and take $I(0)=9.99$, $\eta=8$, $\beta=1,\sigma=0.1$  and $N=10$ as in \cite{chen2021first}. 
In particular we compare to the Lamperti Transformation based methods Lamperti Smoothing Truncation method (\texttt{LST}) of \cite{chen2021first}, the  Lamperti Backward Euler (\texttt{LBE}) scheme of \cite{neuenkirch2014first}, Lamperti Splitting  (\texttt{LS})  \cite{ulander2023boundary} scheme  and the Logarithmic Euler Maruyama (\texttt{LEM}) scheme \cite{yang2021first} for one set of parameter values. We observe that \texttt{LST} and \texttt{LEM} are not yet in the asymptotic regime, even for these small values of $\Dt$. We also see that  \texttt{Proj-EM} shows convergence of order half where as the other methods are essentially equivalent (of order one) and that also \EMWeighted is observed to converge with rate one.

For another set of parameters we compare (see Figure \ref{fig: sis_mid} (b)) to \texttt{ABEM} discussed in Example 1. We take $I(0)=0.95$ and  $\eta=1,\beta=1,\sigma=2$ and $N=1$ as in \cite{ulander2023boundary}.
We observe that \EMWeighted is observed to converge with rate one; the two Euler methods are similar and the Milstein methods are also similar.

\begin{figure}
\begin{center}
    (a) \hspace{0.49\textwidth} (b)\\
\includegraphics[width=0.49\textwidth,height=0.25\textheight]{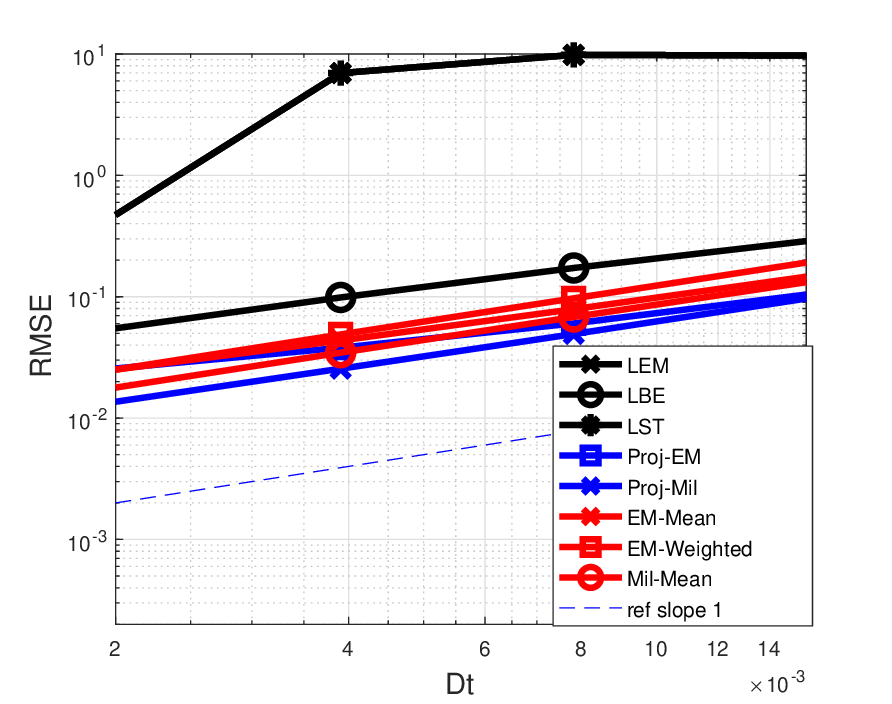}
\includegraphics[width=0.49\textwidth,height=0.25\textheight]{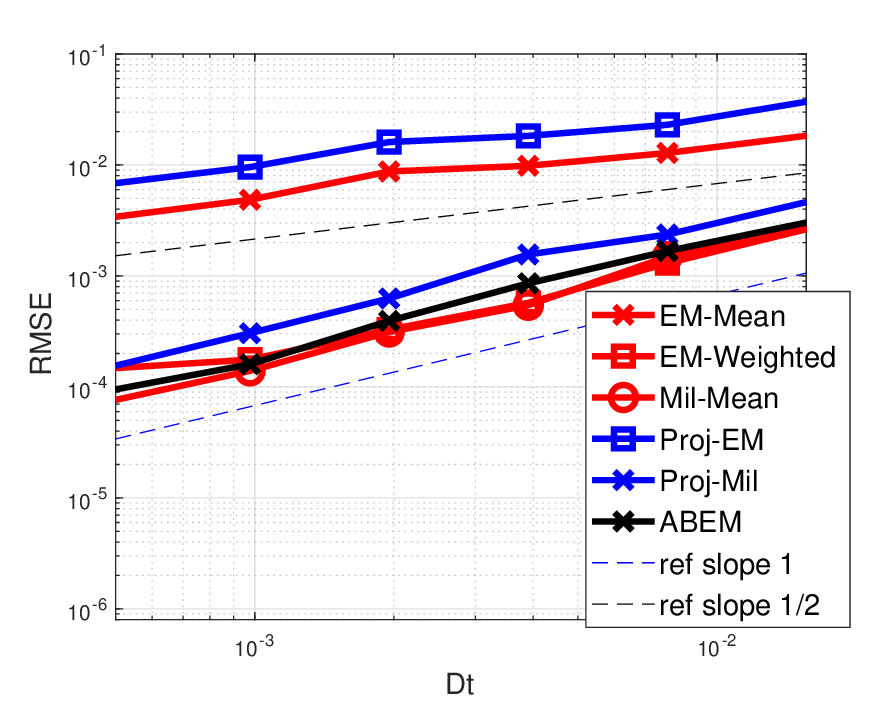}
    \end{center}    
  \caption{RMSE against $\Delta t$ for example 3 with initial condition $X(0)=9.99$ and $N=10$ (a) and $X(0)=0.95$ with $N=1$ (b). See text for other parameter values.}
  \label{fig: sis_mid}
\end{figure}


\subsection{Example 4: Nagumo SPDE}
We consider the standard finite difference approximation in space of the Nagumo type SPDE
  $$dX=\Big[0.001 X_{xx}+X(1-X)(X+\frac{1}{2})\Big]dt+ 2 (1-X)(X+\frac{1}{2}) dW.$$
  with initial condition $ X(x,0)=\Big(1+\exp( -(2-x)/\sqrt{2} ) \Big)^{-1}$ and Neumann boundary conditions with $x\in[0,20]$ and $t\in[0,1]$. We take noise $W$ that is white in space and hence consider space-time white noise.
  For this model we expect $X(x,t)\in[-\frac{1}{2},1]$. From our discretization in space we obtain a system of 128 SDEs. 
  We employ our Euler type methods \EMMean and \MilMean and a standard semi implicit Euler-Maruyama method \cite{lord2014introduction} (denoted \texttt{EM-IMP}) for comparison.
  We observed that \texttt{EM-IMP} does not preserve domain, see for example Figure \eqref{fig: NagumoSPDE} (a) where $\max_{x} X(x,t)$ is plotted for $\Dt=2^{-5}$ for one sample realization of the noise for all three methods. In Figure \eqref{fig: NagumoSPDE} (b) we plot the root mean square error against $\Dt$, based on $1000$ realizations, and observe that \EMWeighted is the most accurate of the methods and we observe a rate of convergence that is better than that for \EMMean.  

\begin{figure}
\begin{center}
    (a) \hspace{0.48\textwidth} (b)\\
  \includegraphics[width=0.48\textwidth,height=0.25\textheight]{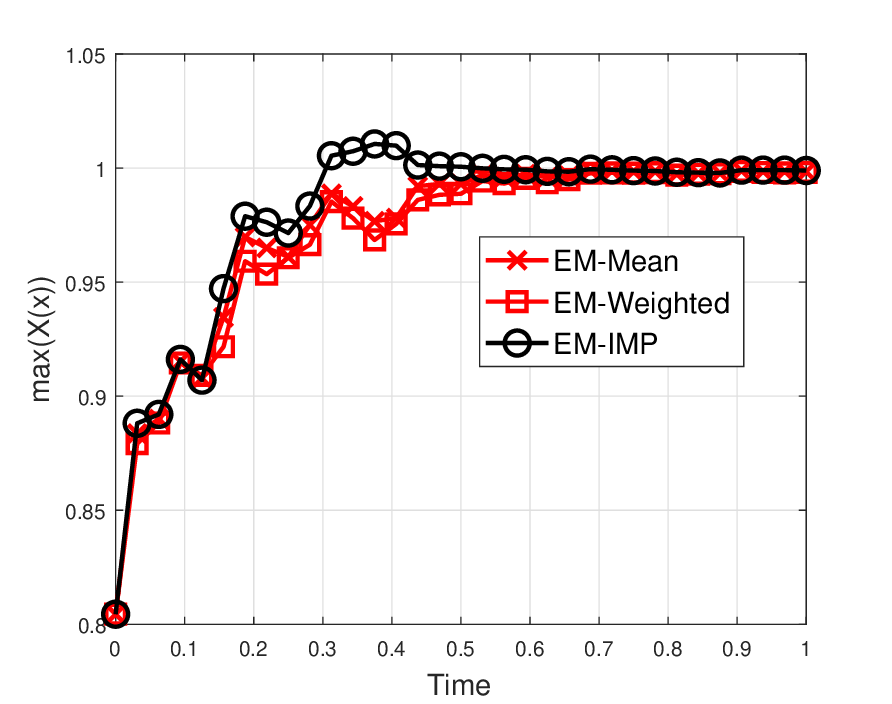}
    \includegraphics[width=0.48\textwidth,height=0.25\textheight]{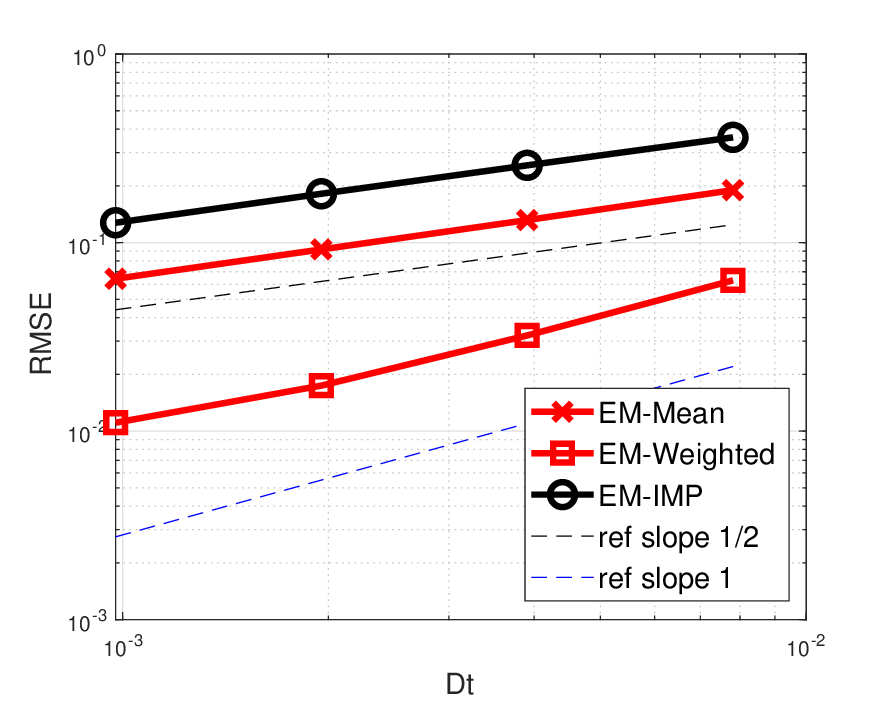}
  \caption{Example 4 (a) $\max_x(X(x,t))$ showing \texttt{EM-IMP} does not preserve the domain and (b) RMSE against $\Dt$.}
  \label{fig: NagumoSPDE}
\end{center}
\end{figure}
  
\printbibliography

%



\end{document}